\documentclass[11pt,oneside]{amsart}
\usepackage{amsmath,ifthen, amsfonts, amssymb, srcltx,
   amsopn, textcomp}

\usepackage{hyperref}
\usepackage{graphicx}
\usepackage{overpic}
\usepackage[small]{caption}

\newcommand{\showcomments}{yes}
\renewcommand{\showcomments}{no}

\newcommand{\hidetodo}[1]
{\ifthenelse{\equal{\showcomments}{yes}}%
{#1}
}

\newsavebox{\commentbox}
\newenvironment{com}%
{\ifthenelse{\equal{\showcomments}{yes}}%
{\footnotemark
        \begin{lrbox}{\commentbox}
        \begin{minipage}[t]{1.25in}\raggedright\sffamily\tiny
        \footnotemark[\arabic{footnote}]}
{\begin{lrbox}{\commentbox}}}%
{\ifthenelse{\equal{\showcomments}{yes}}%
{\end{minipage}\end{lrbox}\marginpar{\usebox{\commentbox}}}
{\end{lrbox}}}

\newtheorem{thm}{Theorem}[section]
\newtheorem{lem}[thm]{Lemma}

\newtheorem{cor}[thm]{Corollary}

\newtheorem{prop}[thm]{Proposition}

\newtheorem*{thmA}{Theorem~\ref{thm:bounded packing abelian cubical}}

\newtheorem*{thmC}{Theorem~\ref{thm:Central HNN virtually special}}
\newtheorem*{thmD}{Theorem~\ref{thm:cocompact cubical flat}}

\newtheorem*{GenericExmp}{Example~\ref{exmp:generic}}

\theoremstyle{definition}
\newtheorem{defn}[thm]{Definition}
\newtheorem{rem}[thm]{Remark}
\newtheorem{exmp}[thm]{Example}

\newtheorem{prob}[thm]{Problem}

\newcommand{\neb}{\mathcal N}

\DeclareMathOperator{\dimension}{dim}

\DeclareMathOperator{\rank}{rank}

\DeclareMathOperator{\link}{link}

\DeclareMathOperator{\stabilizer}{Stabilizer}

\newcommand{\dist}{\textup{\textsf{d}}}

\newcommand{\field}[1]{\mathbb{#1}}
\newcommand{\integers}{\ensuremath{\field{Z}}}

\newcommand{\naturals}{\ensuremath{\field{N}}}
\newcommand{\reals}{\ensuremath{\field{R}}}

\newcommand{\Euclidean}{\ensuremath{\field{E}}}

\newcommand{\boundary}   {{\ensuremath \partial}}

\newcommand{\canon}[2]{\ensuremath{{{\sf C}({#1}\rightarrow {#2})}}}

\DeclareMathOperator{\diameter}{\text{diam}}

\newcommand{\hull}{\text{\sf hull}}

\begin{document}

\title{A cubical flat torus theorem and the bounded packing property}
\author[D.~T.~Wise]{Daniel T. Wise}
           \address{Dept. of Math. \& Stats.\\
                    McGill Univ. \\
                    Montreal, QC, Canada H3A 0B9 }
           \email{wise@math.mcgill.ca}
\author{Daniel J. Woodhouse}
\email{daniel.woodhouse@mail.mcgill.ca}
\date{\today}
\thanks{Research supported by NSERC and the second author by Hydro Quebec.}
\maketitle

\begin{com}
{\bf \normalsize COMMENTS\\}
ARE\\
SHOWING!\\
\end{com}
\begin{abstract}
We prove the bounded packing property for any abelian subgroup of a group acting properly and cocompactly on a CAT(0) cube complex.
A main ingredient of the proof is a cubical flat torus theorem. This ingredient is also used to show that central HNN extensions of maximal free-abelian subgroups of compact special groups are virtually special, and to produce various examples of groups that are not cocompactly cubulated.\end{abstract}

\section{Introduction}

Let $G$ be a finitely generated group, and let $\Upsilon$ be its Cayley graph with respect to some finite generating set.
A subgroup $H \leqslant G$ has \emph{bounded packing in $G$} if for each $r>0$ there exists $m=m(r)$ such that if $g_1H,\ldots, g_mH$ are distinct left cosets of $H$,
then there exists $i,j$ such that $\dist_\Upsilon(g_ih, g_jh')>r$ for all $h,h'\in H$.

The motivating goal of this article is to prove the following:

\begin{thmA}
Let $G$ act properly and cocompactly on a CAT(0) cube complex $\widetilde X$.
Let $A$ be an abelian subgroup of $G$. Then $A$ has bounded packing in $G$.
\end{thmA}

Since Theorem~\ref{thm:bounded packing abelian cubical} is limited to the setting of CAT(0) cube complexes, it offers no direction towards resolving the following problems:
\begin{prob}
\hspace{1cm}\vspace{-.3cm}\newline
\noindent \begin{enumerate}
\item Let $G$ act properly and cocompactly on a CAT(0) space. Does each [cyclic] abelian subgroup $A\leqslant G$ have bounded packing?

\item Let $G$ be an aTmenable group. Does each abelian subgroup $A\leqslant G$ have bounded packing?

\item Find a finitely generated group $G$ with an infinite cyclic subgroup $A \leqslant G$ that does not have bounded packing.
    \end{enumerate}
\end{prob}

The \emph{rank} of a virtually abelian group $A$ is the rank of any finite index free-abelian subgroup of $A$.
A virtually abelian subgroup $A \leqslant G$ is \emph{highest} if $A$ does not have a finite index subgroup that lies in a virtually abelian subgroup of higher rank.
The particular feature of CAT(0) cube complexes used to prove Theorem~\ref{thm:bounded packing abelian cubical} is Theorem~\ref{thm:hullTheorem}, which is the crux of this paper.
A neat consequence of it is the following Cubical Flat Torus Theorem which asserts that
a highest abelian subgroup acts on a product of quasilines.
A \emph{cubical quasiline} is a CAT(0) cube complex that is quasi-isometric to $\mathbb{R}$.
For brevity we will simply refer to cubical quasilines as \emph{quasilines}.
\begin{thmD}
Let $G$ act properly and cocompactly on a CAT(0) cube complex $\widetilde X$.
Let $A$ be a highest virtually abelian subgroup of $G$ and let $p=\rank(A)$.
Then $A$ acts properly and cocompactly on a convex subcomplex $\widetilde Y \subseteq \widetilde X$ such that $\widetilde Y \cong \prod_{i=1}^p C_i$ where each $C_i$ is a quasiline.
\end{thmD}

We also present the following application of Theorem~\ref{thm:bounded packing abelian cubical}:

\begin{thmC}
Let $H$ be a finitely generated virtually $[$compact$]$ special group.
Let $A\subset H$ be a highest abelian subgroup.
Let $G=H*_{A^t=A}$ be the HNN extension, where $t$ is the stable letter commuting with $A$,
then $G$ is virtually $[$compact$]$ special.
\end{thmC}

A version of Theorem~\ref{thm:Central HNN virtually special} was proven in \cite{WiseIsraelHierarchy} under the additional hypothesis of relative hyperbolicity, but Theorem~\ref{thm:bounded packing abelian cubical} allows us to avoid this hypothesis.
Example~\ref{exmp:4 Z2 subgroups} shows that $G$ can fail to have a virtually compact cubulation
when $H$ is a f.g. 2-dimensional right-angled Artin group, but $A$ is not highest.

Section~\ref{sec:restricted intersection} uses Theorem~\ref{thm:cocompact cubical flat} to restrict
how highest abelian subgroups intersect. The following amusing consequence of Corollary~\ref{cor:too many intersection directions} shows that generic multiple cyclic HNN extensions of $\integers^p$ cannot be virtually compactly cubulated:

\begin{GenericExmp}
 Let $\big\{\langle b_1\rangle,\ldots,\langle b_r\rangle,\langle c_1\rangle,\ldots,\langle c_r\rangle \big\}$
 be a collection of  pairwise incommensurable infinite cyclic subgroups of $\integers^p$,
 and suppose that $r>\frac{p}{2}$.
Let $G$ be the following multiple HNN extension of $\integers^p =\langle a_1,\ldots, a_p\rangle$:
$$G = \langle a_1,\ldots, a_p, t_1,\ldots, t_r \ \mid \  [a_i,a_j]=1, \ b_k^{t_k}=c_k \ : \ 1\leq k \leq r\rangle$$
Then $G$ does not contain a finite index subgroup that acts properly and cocompactly on a CAT(0) cube complex.
\end{GenericExmp}

This paper is structured as follows:
In Section~\ref{FlatDual} we prove Theorem~\ref{thm:hullTheorem}.
In Section~\ref{BPproperty} we collect existing results and explain how, along with Theorem~\ref{thm:hullTheorem}, they allow us to prove Theorem~\ref{thm:bounded packing abelian cubical}.
In Section~\ref{sec:restricted intersection} we observe that highest free-abelian subgroups have restricted intersections with other free-abelian subgroups.
In Section~\ref{CentralizingExtensions} we prove Theorem~\ref{thm:Central HNN virtually special}.

\vspace{2mm}
\noindent
{\bf Acknowledgement:}
We are grateful to Mathieu Carette for helpful suggestions.
We thank the referees for their comments and corrections.

\section{The dual to a flat} \label{FlatDual}

The goal of this section is to prove Theorem~\ref{thm:hullTheorem}, which we state below.
 A \emph{quasiline} is a CAT(0) cube complex quasi-isometric to $\mathbb{R}$, and a \emph{quasiray} is a CAT(0) cube complex quasi-isometric to $[0,\infty) \subseteq \mathbb{R}$ .
 A hyperplane in a CAT(0) cube complex $\widetilde X$ will be denoted by $H$, and its left and right halfspaces
 are denoted by $\overleftarrow{H}$ and $\overrightarrow{H}$.
As it is convenient to work with subcomplexes, we define $\overleftarrow{H},\overrightarrow{H}$ to be the smallest subcomplexes containing the complementary components of $\widetilde X-H$.
We will be using wallspaces and Sageev's dual cube complex construction \cite{Sageev95}.
We point the reader to~\cite{HruskaWiseAxioms} for an account of the techniques.
The proof of the following is given at the end of this section, after we have developed the required language and lemmas.

A subset $S \subseteq \widetilde X$ of a geodesic metric space is \emph{convex} if every geodesic with endpoints in $S$ is contained in $S$.
 When $\widetilde X$ is a complete CAT(0) space, a complete connected subspace $Y$ is convex if its inclusion into $\widetilde X$ is a local isometry.
  When $\widetilde X$ is a CAT(0) cube complex, and $Y$ is a subcomplex, there is a simple combinatorial criterion equivalent to being a local isometry: For each $0$-cube $v$ of $Y$ the inclusion $\link_Y(v) \hookrightarrow \link_{\widetilde X}(v)$ is a full subcomplex.
  We refer to~\cite{BridsonHaefliger} for a comprehensive account of CAT(0) metric spaces.

  For a subset $S$ of a CAT(0) cube complex $\widetilde X$, let $\hull(S)$ be the smallest nonempty convex subcomplex of $\widetilde X$ that contains $S$.

\begin{thm}\label{thm:hullTheorem}
 Let $A \leqslant G$ be a virtually abelian subgroup of rank~$p$ that acts properly and cocompactly on a flat $E$ in a CAT(0) cube complex $\widetilde X$ with $\textrm{dim}(\widetilde X) < \infty$.
Then either:
\begin{enumerate}
 \item \label{poss2} $\hull(E)$ is $A$-cocompact and $\hull(E) \cong \prod^{p}_{i=1}C_i$, where each $C_i$ is a convex subcomplex that is a quasiline.
 \item \label{poss1} There exists a finite index subgroup $B \leqslant A$ such that $\min(B) \cap \hull(E)$ is not $B$-cocompact.
\end{enumerate}
\end{thm}

\begin{exmp}
 Consider the cyclic group $A$ generated by a diagonal glide reflection acting on the standard cubulation of the plane $\mathbb{R}^2$.
 Then $\min(A)$ is a diagonal line while $\hull(E)$ is $\mathbb{R}^2$.
\end{exmp}

Let $A \leqslant G$ be a virtually abelian subgroup of rank~$p$ that acts properly and cocompactly on a flat $E$ in a CAT(0) cube complex $\widetilde X$.
By a result of Bieberbach \cite{RatcliffeBook}, there exists a finite index free-abelian subgroup $ A_t \leqslant A$ that acts by translations on $E$.
  Let $P$ be the set of all hyperplanes intersecting $E$.
  The hyperplanes $H_1, H_2$ are \emph{parallel in $E$} if $H_1 \cap E$ and $H_2 \cap E$ are parallel in $E$.
  Being parallel in $E$ is an equivalence relation on the hyperplanes intersecting $E$.
  There are finitely many parallelism classes of hyperplanes in $E$, denoted $P_i \subseteq P$ for $1 \leq i \leq p$.

 \begin{lem}
  There exists a finite index subgroup $B \leqslant A_t$ that acts \emph{disjointly} in the sense that distinct hyperplanes in the same $B$-orbit are disjoint.
 \end{lem}

 \begin{proof}
 For each parallelism class $P_i$, choose $g_i \in A_t$ such that the axis of $g_i$ crosses $H \cap E$ for $H \in P_i$.
There exists $n_i >0$ such that $\langle g_i^{n_i} \rangle$ acts disjointly on the hyperplanes in $P_i$, as otherwise $g_i^jH$ intersects $g_i^kH$ for all $j,k \in \mathbb{Z}$, contradicting that $\dimension(\widetilde X)<\infty$.
 Indeed $m$ pairwise intersecting hyperplanes mutually intersect in an $m$-cube.
  As a CAT(0) cube complex is dual to the wallspace associated to its collection of hyperplanes, a point in $\widetilde X$, together with such a collection of pairwise crossing hyperplanes determines an $m$-cube (see eg~\cite{HruskaWiseAxioms}).
   Alternatively, one can reach a contradiction from Proposition~\ref{prop:helly} applied to the subdivision of $\widetilde X$.

For $\epsilon>0$ to be determined below, we let $E^\epsilon$ denote $\neb_\epsilon(E)$ and for a hyperplane $H\in P_i$ we let $H^\epsilon = H\cap E^\epsilon$ and let $P_i^\epsilon = \{H^\epsilon: H\in P_i\}$.
Thus $(E^\epsilon, P_i^\epsilon)$ is a wallspace for each $i$.
We choose $\epsilon$ so that for each $i$, for each pair $H,H' \in P_i$
we have $H\cap E = H'\cap E$ if and only if $H,H'$ cross within $E^\epsilon$.
Cocompactness of $E$ ensures that there are finitely many intersection angles between
hyperplanes and $E$, and this allows us to bound the number of orbits of codimension-2 hyperplanes intersecting $\neb_1(E)$ but not intersecting $E$. We choose $\epsilon$ to be less than the minimal distance from $E$ to any such codimension-2 hyperplane.

 The dual cube complex $C(E^\epsilon, P_i^\epsilon)$ is a quasiline with an $A_t$-action.
 The kernel $K_i$ of this action is isomorphic to $\integers^{p-1}$.
 Adjoining $g_i^{n_i}$ to $K_i$, we obtain the
finite index subgroup $B_i = \langle K_i , g_i^{n_i} \rangle \leqslant A_t$.
For each $H\in P_i$ we have $B_iH = \langle g_i^{n_i} \rangle H$ and hence $B_i$ acts disjointly on hyperplanes in $P_i$.
Finally, $B = \bigcap_i B_i$ acts disjointly on the hyperplanes in $P$.
 \end{proof}

  For each parallelism class $P_i$, let $Z_i \leqslant B$ be an infinite cyclic group stabilizing a line $R_i \subseteq E$, non parallel to $H \cap E$ for all $H \in P_i$.
 Consider two hyperplanes $H,H' \in P_i$ with distinct $Z_i$-orbits:
 $\{z_i^jH: j\in \integers\}$ and $\{z_i^jH'  : j\in \integers\}$.
 Observe that if $H$ intersects both $z_i^jH'$ and $z_i^kH'$, then $H$ intersects $z_i^\ell H'$ for $j\leq \ell \leq k$.

\begin{defn}
We say  $H,H' \in P$ represent \emph{crossing} orbits if either $H \in P_i$ and $H' \in P_j$ with $i \neq j$, or if $H,H' \in P_i$ and $z_i^{j}H$ crosses $z_i^kH'$ for all $j,k\in \integers$.
We say $H, H' \in P_i$ represent \emph{aligned} orbits if $H$ intersects only finitely many $Z_i$-translates of $H'$.
We say  $H,H' \in P_i$ represent \emph{semi-crossing} orbits if they do not represent crossing orbits, and there exists $m \in \integers$ such that either $z_i^jH$ crosses $z_i^kH'$ for $j - k > m$, or $z_i^jH$ crosses $z_i^kH'$ for $k-j > m$.
Any pair of hyperplanes in $P$ must represent either  crossing,  semicrossing, or aligned orbits.
\end{defn}

\begin{figure}
  \begin{overpic}[width=.45\textwidth,tics=10]{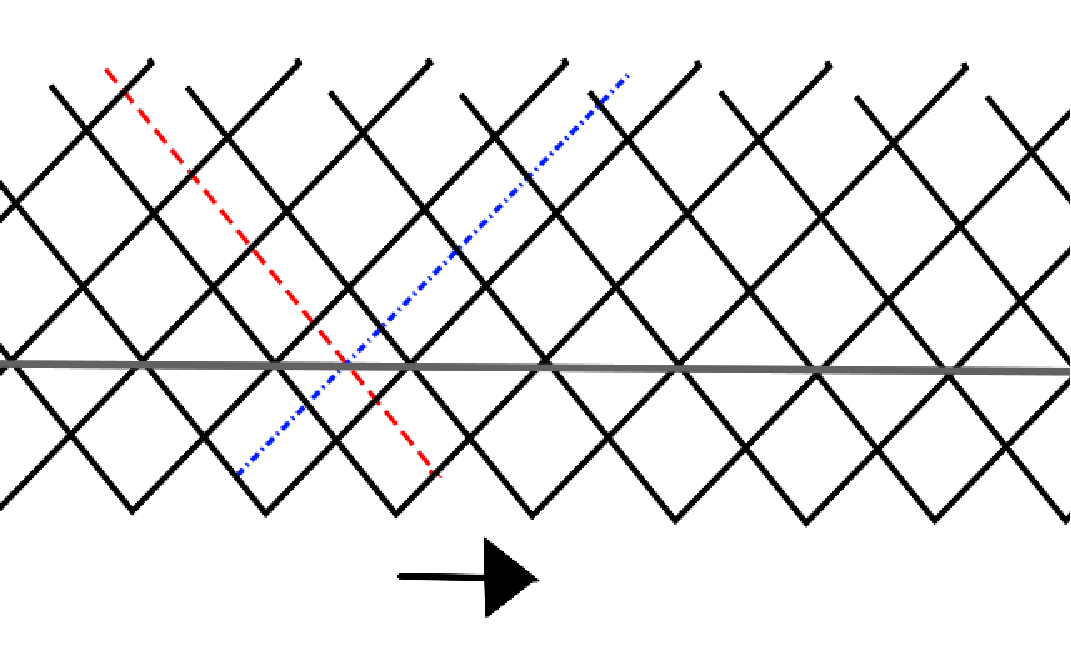}
  \end{overpic}
  \caption{\label{diag:Halfplane}A cubical halfplane with two semi-crossing hyperplane orbits.}
\end{figure}

\begin{exmp}
Let $\mathbb{Z}$ act on the cubical halfplane in Figure~\ref{diag:Halfplane}.
The flat plane $E$ is the line homeomorphic to $\mathbb{R}$.
Observe that there are two parallelism classes of orbits that semi-cross.
However, the action is not cocompact.
\end{exmp}

\begin{lem}
 Alignment of $Z_i$-orbits is an equivalence relation.
\end{lem}

\begin{proof}
 Reflexivity and symmetry are immediate.
Suppose $Z_iH$ is aligned to $Z_iH'$ and $Z_iH'$ is aligned to $Z_iH''$,
but infinitely many hyperplanes in $Z_iH$ cross $H''$.
 By their alignment, there exists $j,k$ such that $H'' \subset (z_i^j \overrightarrow{H}' \cap z_i^k \overleftarrow{H}')$.
 Since infinitely many elements of $Z_iH$ intersect $H''$, infinitely many of these elements intersect $H''$
  but do not intersect $z_i^j H'$ or $z_i^k H'$.
 This is a contradiction as infinitely many hyperplanes cannot separate $z_i^j H'$ and $z_i^k H'$.
 \end{proof}

\begin{lem} \label{partialOrdering}
Semi-crossing of $Z_i$-orbits is a partial ordering, denoted by $>$, where $Z_iH > Z_iH'$ if $Z_iH$ and $Z_iH'$ are semi-crossing and there exists $m \in \integers$ such that $z_i^jH$ crosses $z_i^kH'$ for $j - k > m$.
\end{lem}

\begin{proof}
Antisymmetry holds since, if $Z_iH > Z_iH'$ and $Z_iH' > Z_iH$ but $Z_iH \neq Z_iH'$ then $Z_iH$ and $Z_iH'$ are crossing orbits. However, distinct crossing orbits are not semi-crossing.

To prove transitivity we first prove the following claim: If $Z_iH_1$ and $Z_iH_2$ are aligned orbits and $Z_iH' > Z_iH_1$, then $Z_iH' > Z_iH_2$.
By their alignment, there exists $p<q$ such that $H_2 \subset (z_i^p \overrightarrow{H}_1 \cap z_i^q \overleftarrow{H}_1)$.
   Suppose  $z_i^j H'$ crosses $z_i^kH_1$ precisely for $j-k > m$.
    Then $H_2$ is crossed by $z_i^jH'$ for $j - q > m$ as $z_i^jH'$ crosses both $z_i^qH_1$ and $z_i^pH_1$. Similarly $H_2$ is not crossed by $z_i^jH'$ for $j-p \leq m$. This implies that $Z_iH' > Z_iH_2$.
  Similarly $Z_iH' < Z_iH_1$ would imply $Z_iH' < Z_iH_2$.

  Suppose that $Z_iH_3 > Z_iH_2 > Z_iH_1$.
  By the claim, $Z_iH_1$ cannot be aligned to $Z_iH_3$.
  Therefore we need only exclude the possibility that $Z_iH_1$ and $Z_iH_3$ are crossing orbits.
Since $Z_iH_2> Z_iH_1$ there exists $N_1$ such that $z_i^nH_2$ is disjoint from $H_1$ for all $n \leq N_1$. Assume that $H \subseteq z_i^{N_1}\overrightarrow{H}_2$.
  Since $Z_iH_3>Z_iH_2$ there exists $N_2$ such that $z_i^nH_3$ is disjoint from $z_i^{N_1}H_2$ for all $n \leq N_2$.
  Since $z_i$ acts by translation on $E$ we can deduce that $z_i^nH_3 \subseteq z_i^{N_1}\overleftarrow{H}_2$ for $n \leq N_2$.
  Hence, $z_i^nH_3$ is disjoint from $H_1$ for all $n \leq N_2$ as they are separated by $z_i^{N_1}{H}_2$.
 \end{proof}

 \begin{proof}[Proof of Theorem~\ref{thm:hullTheorem}]
We first assume there are no semi-crossing orbits.
In the next three steps we will show that $\hull(E) = \prod_{i=1}^p C_i$  where each $C_i$ is the quasiline dual to the family of hyperplanes corresponding to an alignment class.

First, we claim that $\hull(E)$ is isomorphic to $C(\widetilde X, P)$.
Indeed, each $0$-cube $x$ in $\hull(E)$ corresponds to the $0$-cube $y$ in $C(\widetilde X, P)$ where each hyperplane in $P$ is oriented towards $x$.
Conversely, a $0$-cube $y$ in $C(\widetilde X, P)$ corresponds to a $0$-cube $x$ in $\widetilde X$ by orienting the hyperplanes not in $P$ towards $E$.
 Moreover, $x \in \hull(E)$ since $x$ lies in each halfspaces containing $E$.
This bijection preserves adjacency.

 Secondly, let $\{ A_i \}_{i=1}^p$ be an enumeration of the alignment classes.
 Observe $C(\widetilde X, P) \cong \prod_{i=1}^m C(\widetilde X, A_i)$ as every hyperplane in $A_i$ intersects every hyperplane in $A_j$ for $i \neq j$.
 Indeed, a $0$-cube of $C(\widetilde X, P)$ determines a $0$-cube of each of the factors by ignoring the orientations in the other alignment classes.
 Conversely, a choice $0$-cubes in each of the factors determines a $0$-cube in $C(\widetilde X, P)$ since the hyperplanes cross each other.
 Again, it is easy to see this bijection preserves adjacency.

 Thirdly, let $G_i = \langle g_i \rangle$ be an infinite cyclic subgroup of $B$ acting freely on $C(\widetilde X, A_i)$.
 We will show that $C(\widetilde X, A_i)$ is $G_i$-cocompact, and therefore quasiisometric to $\reals$.
 Let $H_1,  \ldots, H_k$ be representatives of the distinct $G_i$-orbits.
 Note that the dimension of $C(\widetilde X, A_i)$ is bounded by $k$.
 We now show that there are finitely many $G_i$-orbits of maximal cubes.
 A maximal cube corresponds to a collection of pairwise intersecting hyperplanes $g_i^{\alpha_1}H_{j_1}, \ldots g_i^{\alpha_\ell}H_{j_\ell}$.
 By translating we can assume that $\alpha_1 = 0$, and therefore there are finitely many such collections since only finitely many hyperplanes can intersect $H_{j_1}$.

  Then $\hull(E) = \prod_{i=1}^p C_i$  where each $C_i$ is the quasiline dual to the family of hyperplanes corresponding to an alignment class.
 Observe that $B$ acts by translations on $E$ with disjoint hyperplane-orbits and hence stabilizes each alignment class and thus preserves the factors of the product structure.
 If $p= \rank(A)$ the action on $\hull(E)$ is cocompact, which implies that the set of hyperplanes orthogonal to each $R_i$ belong to a single alignment class.
 Otherwise $p>\rank(A)$, and as $B$ acts metrically properly and cocompactly on $C_i$,
 each $C_i$ contains an isometrically embedded $B$-invariant line $\ell_i$.
  Thus $\prod \ell_i \subseteq \prod_{i=1}^p C_i$ is not cocompact,
 but is contained in $\min(B) \cap \hull(E)$.

Suppose there are at least two semi-crossing orbits in some parallelism class,
and let $Q$ be a maximal alignment class with respect to the partial ordering.
For each parallelism class $P_i$ and orbit $Z_iH \subseteq P_i - Q$:
 either $Z_iH$ crosses the orbits in $Q$,
or $Q\subset P_i$ and $Z_iH' > Z_iH$ for all $H' \in Q$.

We define a sequence of $B$-equivariant cubical maps
 $\{\phi_k: \hull(E) \rightarrow \hull(E)\}_{k\in \naturals}$
 using the partition $P = Q \sqcup Q^c$:
A $0$-cell $x$ in $\hull(E)$ corresponds uniquely to a choice of orientation for each hyperplane intersecting $E$.
 Let $x[H] \in \big\{\overleftarrow{H}, \overrightarrow{H}\big\}$ denote the halfspace of $H$ containing $x$ in its interior. Its image $\phi_k(x)$ is specified by how $\phi_k(x)$ orients the hyperplanes intersecting $E$.
 For $H \in Q^c$ let $\phi_k(x)[H] = x[H]$.
 For $H \in Q \subseteq P_i$ let $\phi_k(x)[z_i^{k}H] = x[H]$.
 This defines a $0$-cube in $\hull(E)$ since only finitely many hyperplanes have their orientations changed, and disjoint hyperplanes are not oriented away from each other by $\phi_k(x)$:
 Let $H \subseteq P_i \subseteq Q^c$ represent a $Z_i$-orbit not crossing the $Z_i$-orbits in $Q \subseteq P_i$, then $Z_iH' > Z_iH$ for any $H' \in Q$. Therefore, if $H'$ crosses $H$ then $z_i^{k}H'$ also crosses $H$.

 The injectivity of $\phi_k$ on $0$-cubes holds since if $x_1\neq x_2$ then there exists $H \in P$ such that $x_1[H] \neq x_2[H]$.
 If $H \in Q^c$ then $\phi_k(x_1)[H] = x_1[H] \neq x_2[H] = \phi_k(x_2)[H]$ so $\phi_k(x_1) \neq \phi_k(x_2)$.
 If $H \in Q$ then $\phi_k(x_1)[z_i^{-k}H] = x_1[H] \neq x_2[H] = \phi_k(x_2)[z_i^{-k}H]$ so $\phi_k(x_1) \neq \phi_k(x_2)$.
 Therefore $\phi_k$ is injective on the $0$-skeleton.
 Similar reasoning shows that $\phi_k$ sends adjacent $0$-cubes to adjacent $0$-cubes and so $\phi_k$ extends to the $1$-skeleton of $\hull(E)$.
 Moreover injective maps on the 1-skeleton send squares to squares, hence the map also extends to the 2-skeleton.

 Any map defined on the 2-skeleton of a cube complex extends uniquely to a cubical map on the entire complex.
 Observe that $B$ acts on $E$ by translation and preserves each $Z_j$-orbit in each $P_j$.
 Therefore, for each $b \in B$ there exists $\ell_i$, for $1 \leq i \leq p$, such that $b H = z_i^{\ell_i}H$ for each $H \in P_i$.
 Therefore $\phi_k$ is $B$-equivariant since if $H \in P_i$ but $H \notin Q$ then
\begin{equation*}\begin{split}
(b\cdot \phi_k(x))[H]  &= \phi_k(x)[b^{-1}  H]  =  \phi_k(x)[z_i^{-\ell_i} H]    \\
 &= x[z_i^{-\ell_i}H]  = x[b^{-1} H]  =  (b \cdot x)[H]  = \phi_k(b\cdot x)[H].
   \end{split}
   \end{equation*}
 \noindent Similarly, if $H \in Q \subseteq P_i$ then
\begin{equation*}\begin{split}
(b\cdot \phi_k(x))[H] &= \phi_k(x)[b^{-1}  H] = \phi_k(x)[z_i^{-\ell_i} H] \\
&= x[z_i^{k-\ell_i}H] = x[b^{-1}z_i^{k} H] = (b \cdot x)[z_i^kH] = \phi_k(b\cdot x)[H].
\end{split}\end{equation*}

 We now show that $\dist(\phi_k(x),b\cdot x)\geq k$ for each $b \in B$ and $x$ a canonical $0$-cube $x$ associated to a point in $E$.
 For each $Z_j$-orbit $Z_jH$ fix a representative $H$ such that $x \in \overrightarrow{H} \cap z_i \overleftarrow{H}$.
 Let $H \in Q \subseteq P_i$ be such a representative, then $\phi_k$ changes the orientation of precisely $k$ hyperplanes in $Z_iH$, namely $z_iH , \ldots, z_i^kH$.
 For any representative $H \in P_i$, however, translation by $b$ changes the orientation of $\ell_i$ hyperplanes, namely $z_iH, \ldots, z_i^{\ell_i}H$.
 As there is at least one $Z_i$-orbit in $P_i \supseteq Q$ not in $Q$, we can deduce that at least $k$ hyperplanes have distinct orientations in $\phi_k(x)$ and $b\cdot x$.
 Therefore the distance from $bx$ to $\phi_k(x)$ is at least $k$.

Observe that $\dist(\phi_k(y_1),\phi_k(y_2)) \leq \dist(y_1,y_2)$ for $y_1,y_2\in \hull(E)$.
Indeed, the CAT(0) metric on $\hull(E)$ is defined to be the infimal length of piecewise Euclidean paths joining points, and the map preserves lengths of paths.
The $B$-equivariance together with that $\phi_k$ is distance-nonincreasing implies that $\phi_k(e)\in \min(B)$ for each $e\in E$.

 In conclusion, $\dist(\phi_k(E), E)) \rightarrow \infty$ as $k\rightarrow \infty$.
For $e \in E$, the orbit $Be \subseteq E$ is mapped by a distance-nonincreasing function to a new orbit at distance $ \geq k$ from $E$.
Since the original flat was in $\min(B)$, the image of the image orbit is isometric to the original orbit.
Since $k$ is unbounded, $\min(B) \cap \hull(E)$ is non-cocompact.
 \end{proof}

\section{The bounded packing property} \label{BPproperty}

Let $G$ be a finitely generated group with Cayley graph $\Upsilon$.
Suppose $G$ acts by isometries on a geodesic metric space $\widetilde{X}$ such that the map $g \mapsto gx_0$ is a quasi-isometric embedding for some $x_0 \in \widetilde{X}$.
Then $H \leqslant G$ has bounded packing if and only if for each $r >0$ there exists $m = m(r)$ such that if $g_1H, \ldots , g_mH$ are distinct left cosets of $H$, then there exists $i,j$ such that $\dist_{\widetilde{X}}(g_iHx_0, g_jHx_0) > r$.
We refer to \cite{HruskaWisePacking} for more about bounded packing, and specifically to Cor~2.9, Lem~2.3, and Lem~2.4 for the following:

\begin{lem} \label{lem:finite index subnormal}
Suppose $H$ has bounded packing in $G$.
If $K \leqslant H$ is a normal subgroup of $H$, then $K$ has bounded packing in $G$.
If $K \leqslant G$ is a subgroup such that $[H:K\cap H]<\infty$ and $[K:K\cap H]<\infty$, then $K$ has bounded packing.
\end{lem}

A proof of the following well known fact follows from the median space structure of the 1-skeleton of a CAT(0) cube complex ~\cite[Thm~2.2]{RollerPocSets}.
A proof using disk diagrams can be found in~\cite[Sec 2]{WiseIsraelHierarchy}.

\begin{prop}[Helly Property]\label{prop:helly}
Let $Y_1,\ldots, Y_r$ be convex subcomplexes of a CAT(0) cube complex.
If $Y_i\cap Y_j\neq \emptyset$ for each $i,j$,
then $\cap_{i=1}^r Y_i \neq \emptyset$.
\end{prop}

The following is obtained in  \cite[Lem~13.15]{HaglundWiseSpecial}:

\begin{lem}\label{lem:r thickening}
Let $\widetilde Y\subset \widetilde X$ be a convex subcomplex of the CAT(0) cube complex $\widetilde X$.
For each $r\geq 0$ there exists a convex subcomplex $\widetilde Y^{+r}$ such that $\neb_r(\widetilde Y)\subset \widetilde Y^{+r} \subset \neb_s(\widetilde Y)$
for some $s\geq 0$.
Here $\neb_m(\widetilde Y)$ denotes the $m$-neighborhood of $\widetilde Y$.
\end{lem}

 We infer the following from the above results.

\begin{lem}\label{lem:convex BP}
Let $G$ act properly and cocompactly on a CAT(0) cube complex $\widetilde X$.
Let $H$ be a subgroup that cocompactly stabilizes a nonempty convex subcomplex $\widetilde Y\subset \widetilde X$.
Then $H$ has bounded packing in $G$.
\end{lem}
\begin{proof}
Let $x_o \in \widetilde Y$.
Let $g_1, g_2, \ldots$ be an enumeration of the left coset representatives of $H$.
Let $\widetilde Y^{+r}$ be as in Lemma~\ref{lem:r thickening}.
Observe that if $\dist(g_jHx_o, g_kHx_o)< r$ then $g_j \widetilde Y^{+r} \cap g_k \widetilde Y^{+r} \neq \emptyset$.
Thus, to show that $H$ has bounded packing, it suffices to find an upper bound on the number of distinct
cosets $g_iH$ such that $\{g_i\widetilde Y^{+r}\}$ pairwise intersect.
Moreover,
by Proposition~\ref{prop:helly}, if $\{g_1\widetilde Y^{+r}, \ldots, g_p\widetilde Y^{+r} \}$ pairwise intersect then  $\cap_{i=1}^p g_i\widetilde Y^{+r}\neq \emptyset$.
It thus suffices to show that there is an upper bound $m$ on the multiplicity of
$\{g_i\widetilde Y^{+r} : g_iH \in G / H \}$. However this collection of sets is uniformly locally finite since
 $\widetilde Y^{+r}$ is $H$-cocompact and $[\stabilizer(Y^{+r}):H]< \infty$.
\end{proof}

If $A$ is an abelian group acting by isometries on a metric space $\widetilde{X}$, then $\min(A)$ is the set of all $\widetilde{x} \in \widetilde{X}$ such that $\dist(a\widetilde{x}, \widetilde{x}) \leq \dist(a\widetilde{y}, \widetilde{y})$ for all $a \in A$ and $\widetilde{y} \in \widetilde{X}$.
We refer to \cite{BridsonHaefliger} for the following:

\begin{prop}[Flat torus theorem]\label{prop:FPT}
Let $A$ be a virtually free-abelian group of rank~$n$ acting metrically properly and semisimply on a CAT(0) space $\widetilde X$.
There exists a subspace $V\times F \subset \widetilde X$ with $F$ isometric to $\Euclidean^n$
such that  $A$ stabilizes $V\times F$ and acts as: \  $a(v,f)=(v,af)$ for all $(v,f)\in V\times F$ and $a \in A$.
Moreover, if $A \cong \mathbb{Z}^n $ then $ \min(A)= V\times F$.
\end{prop}

\begin{thm}[Cubical flat torus theorem]\label{thm:cocompact cubical flat}
Let $G$ act properly and cocompactly on a CAT(0) cube complex $\widetilde X$.
Let $A$ be a highest virtually abelian subgroup of $G$ and let $p=\rank(A)$.
Then $A$ acts properly and cocompactly on a convex subcomplex $\widetilde Y \subseteq \widetilde X$ such that $\widetilde Y \cong \prod_{i=1}^p C_i$ where each $C_i$ is a quasiline.
\end{thm}

\begin{proof}
By Proposition~\ref{prop:FPT}, $A$ stabilizes $E\subset \widetilde X$, where $E$ is isometric to $\Euclidean^p$.
By Theorem~\ref{thm:hullTheorem}, either $\hull(E) \cong \prod_{i=1}^p C_i$ is $A$-cocompact where each $C_i$ is a quasiline, or there exists a finite index free-abelian subgroup $B \leqslant A$ such that $\min(B) \cap \hull(E)$ is not $B$-cocompact.
  We shall show that the second possibility contradicts that $A$ is highest.

Applying Proposition~\ref{prop:FPT} again, let $\min(B) =V\times F$, where $\diameter(V) = \infty$.
For $v\in V$ let $N(\{v\}\times F)$ denote the smallest $B$-invariant connected subcomplex of $\widetilde X$ containing $\{v\}\times F$.
Since $\{v\}\times F$ is $B$-cocompact, so is $N(\{v\}\times F)$.
Moreover, the number of $B$-orbits of cells in $N(\{v\}\times F)$  is bounded by a constant independent of $v\in V$.
Indeed, by $B$-cocompactness,  there is $m>0$ such that $F = B\neb_m(f )$ for each $f \in F$.
For each $v$ there is a $B$-equivariant isometry $F \rightarrow \{v\}\times F$, and so $\{v\}\times F$ is likewise covered by the $B$ translates of each $m$-ball.
However, the number of cells intersecting an $m$-ball in $\widetilde X$ is finite by properness and cocompactness.
So the number of $B$-orbits of cells in $N(\{v\}\times F)$ has the same upper bound.

It follows that there are finitely many $G$-orbits of subcomplexes $N(\{v\}\times F)$.
As $\diameter(V)=\infty$, there are points $v_1,v_2\in V$ and $g\in G$ such that $N(\{v_1\}\times F)\cap N(\{v_2\}\times F)=\emptyset $, but
$gN(\{v_1\}\times F)=N(\{v_2\}\times F)$.
Both $B$ and $g$ stabilize $\sqcup g^n N(\{v_1\}\times F)$ which is quasi-isometric to $\Euclidean^{n+1}$. Hence $\langle g, B\rangle$ is a higher rank virtually abelian subgroup.
\end{proof}

\begin{thm}\label{thm:bounded packing abelian cubical}
Let $G$ act properly and cocompactly on a CAT(0) cube complex $\widetilde X$.
Let $A$ be an abelian subgroup of $G$. Then $A$ has bounded packing in $G$.
\end{thm}
\begin{proof}
By Proposition~\ref{prop:FPT} and the assumption that $\widetilde{X}$ is finite dimensional we can find a highest virtually free-abelian group $A'$ that contains a finite index subgroup of $A$.
The result now follows by combining Lemmas~\ref{lem:finite index subnormal},~\ref{lem:convex BP}, and Theorem~\ref{thm:cocompact cubical flat}.
\end{proof}

\section{Subproduct intersections} \label{sec:restricted intersection}
This section illustrates the following consequence of Theorem~\ref{thm:cocompact cubical flat}:

\begin{thm} \label{thm:highestIntersectionSubgroups}
 Let $G$ act properly and cocompactly on a CAT(0) cube complex $\widetilde{X}$.
 Let $A \leqslant G$ be a highest free-abelian subgroup, and let $p=\rank(A)$.
 There is a set $S = \{\hat a_1, \ldots, \hat a_p\} \subseteq A$ such that the following holds:
  For any highest free-abelian subgroup $A' \leqslant G$, the intersection
   $A'\cap A$ is commensurable to a subgroup generated by a subset of $S$.
\end{thm}

Proving Theorem~\ref{thm:highestIntersectionSubgroups} requires the following consequence of the flat torus theorem.

\begin{lem} \label{lem:actionDecomposition}
 Let $A$ be a rank~$p$ virtually abelian group acting properly and cocompactly on a CAT(0) cube complex $\prod_{i=1}^p C_i$, where each $C_i$ is a quasiline.
 Then there exists a finite index free-abelian subgroup $\hat{A} \leqslant A$ with basis $\{\hat{a}_1,\ldots, \hat{a}_p\}$ such that $\hat{a}_i \cdot (c_1, \ldots , c_i, \ldots, c_p) = (c_1 ,\ldots, \hat{a}_i \cdot c_i, \ldots c_p)$ for each $i$.
\end{lem}

\begin{proof}
 The action of $A$ on $\prod_{i=1}^p C_i$ permutes the factors in the product, yielding a homomorphism $A \rightarrow S_p$ to the degree~$p$ symmetric group.
 Its kernel is a finite index subgroup $B \leqslant A$ such that the $B$-action on $\prod_{i=1}^p C_i$
 is the product of $B$-actions on the factors.
  For each $i$ there is a finite index subgroup $B_i\leqslant B$ that acts by translations on an invariant line $\ell_i \subset C_i$.
 Let $\hat A = \bigcap_{i=1}^p B_i$.
 Consider a homomorphism $\phi : \hat A \rightarrow \integers^p$ induced by the action of $\hat A$ on $\prod_{i=1}^p \ell_i$.
 Since $\hat A$ acts cocompactly on $\prod_{i=1}^p \ell_i$ we deduce that $[\integers^p : \phi(\hat A)]<\infty$.
 Therefore, there are $\hat{a}_i \in \hat A$ such that $\phi(\hat{a}_i) = (0, \ldots , 0 , m_i, 0 ,\ldots , 0)$, where $m_i \neq 0$ is the $i$-th entry.
\end{proof}

 We earlier defined the halfspaces $\overleftarrow{H}, \overrightarrow{H}$ associated to a hyperplane $H$ of $X$ to be the smallest subcomplexes containing the components $X - H$.
The \emph{small halfspaces} are the largest subcomplexes contained in the two components of $X-H$.
 Equivalently, the small halfspaces are the components of $X - N^o(H)$, where $N^o(H)$ is the union of open cubes intersecting $H$.
 Note that each small halfspace is convex as each component of $\boundary N^o(H)$ is convex.
 It is readily verified that a subcomplex of $X$ is convex if and only if it is the intersection of small halfspaces.

\begin{lem} \label{lem:convexSubproduct}
Let $\widetilde X = \prod \widetilde X_i$ where each $\widetilde X_i$ is a connected CAT(0) cube complex.
Then a convex subcomplex $\widetilde Y \subseteq \widetilde X$ is a product $\widetilde Y = \prod \widetilde Y_i$, where $\widetilde Y_i \subseteq \widetilde X_i$ is a convex subcomplex.
\end{lem}

\begin{proof}
 Let $\widetilde Y$ be a convex subcomplex of $\widetilde X$.
 Each $1$-cube is the product of some $0$-cubes and a single $1$-cube in some factor $\widetilde X_i$.
 An $\widetilde X_i$ hyperplane is a hyperplane which is dual to a $1$-cube arising from a factor $\widetilde X_i$.
 Let $\widetilde Y_i$ be the intersection of all small halfspaces containing $\widetilde Y$ that are associated to $\widetilde X_i$ hyperplanes.
 Then it is immediate that $\widetilde Y = \prod \widetilde Y_i$.
\end{proof}

\begin{proof}[Proof of Theorem~\ref{thm:highestIntersectionSubgroups}]
 By Theorem~\ref{thm:cocompact cubical flat},
  $A$ acts properly and cocompactly on a convex subcomplex $\widetilde{Y} \cong \prod_{i=1}^p C_i \subseteq \widetilde{X}$.
  A halfspace is \emph{shallow} if it lies in a finite neighborhood of its hyperplane,
  and is \emph{deep} otherwise.
  By passing to a smallest nonempty convex $A$-invariant subcomplex of $\widetilde Y$, we may assume
  that no hyperplane in $\widetilde Y$ has both a shallow and a deep halfspace.
  The convex hull of an $A$-invariant $p$-flat $F \subseteq \widetilde Y$ has this property.
  Indeed, each hyperplane intersecting $F$ in a $(p-1)$-flat necessarily has a pair of deep halfspaces, and a hyperplane containing $F$ has two shallow halfspaces by cocompactness.
By Lemma~\ref{lem:actionDecomposition}, there is a finite index subgroup $\hat{A}=\prod_{i=1}^p \langle \hat a_i \rangle$ of $A$ such that $\langle \hat{a}_i \rangle$ acts cocompactly on $C_i$, and trivially on $C_j$ for $j \neq i$.
Let $S = \{ \hat{a}_1, \ldots, \hat{a}_p \}$.
 Similarly, $A'$ cocompactly stabilizes a convex subcomplex $\widetilde{Y}'\subseteq \widetilde{X}$ which has its own induced product decomposition, and there exists a corresponding finite index subgroup $\hat A' = \prod_{i=1}^{p'} \langle \hat a_i' \rangle$ that acts cocompactly on $\widetilde Y'$.

By Lemma~\ref{lem:r thickening} for each $r$ there exists a cubical $r$-thickening $(\widetilde Y')^{+r}$
containing $\neb_r(\widetilde Y')$ and $(\widetilde Y')^{+r}$ is convex and $\hat{A}'$-cocompact.
Choose $r$ so that $\widetilde Y \cap (\widetilde Y')^{+r} \neq \emptyset$ and note that $\widetilde Y \cap (\widetilde Y')^{+r}$ is also convex. Therefore, by Lemma~\ref{lem:convexSubproduct}, the intersection is a subproduct $\widetilde Y \cap (\widetilde Y')^{+r} \subseteq \prod D_i \subseteq \prod C_i$ where each $D_i\subset C_i$ is a convex subcomplex.
Thus each factor is either a quasiline, a quasiray, or a compact convex subcomplex.
Furthermore, the action of $\hat{A} \cap \hat{A}'$ on $\widetilde Y \cap (\widetilde Y')^{+r}$ is cocompact.
Indeed, the intersection $\widetilde Y \cap (\widetilde Y')^{+r}$ is the universal cover of a component of the fiber product of $\hat{A} \backslash \widetilde Y  \rightarrow G \backslash \widetilde{X}$ and $\hat{A}' \backslash (\widetilde Y')^{+r} \rightarrow G \backslash \widetilde X$.

For each $i$, if $D_i$ is a quasiline or compact then let $E_i=D_i$, and otherwise let $E_i$ be the compact, $\hat A \cap \hat{A}'$-invariant subcomplex contained in the intersection of all shallow halfspaces of $D_i$ that have deep complements.
 Note that $D_i$ is nonempty since by finite dimensionality, $D_i$ is the intersection of finitely many shallow halfspaces whose associated hyperplanes intersect, and thus the Helly property implies the intersection is nonempty.
Let $E=\prod E_i$.
 If $E_i$ is a quasiline, then $\stabilizer_{\hat{A}}(E_i) = \langle \hat a_1 , \ldots,  \hat a_i^{n_i}, \ldots \hat a_p \rangle$ for some $n_i > 0$ since $\hat A \cap \hat{A}'$ must act cocompactly on $E$.
 Otherwise, if $E_i$ is compact, then $\stabilizer_{\hat{A}}(E_i) = \langle \hat a_1 ,\ldots , \hat a_{i-1}, \hat a_{i+1}, \ldots, \hat a_p \rangle$.
 Let $S_o \subseteq S$ be the subset of $S$ such that $i \in S_o$ if $E_i$ is a quasiline.
 Therefore $\stabilizer_{\hat{A}}(E)$ acts cocompactly on $E$, is commensurable to the subgroup generated by $S_o$, and contains $\hat{A} \cap \hat{A}'$.

Assume now that $r$ is large enough that $(\widetilde Y)^{+r}\cap \widetilde Y'\neq \emptyset$ and
as before $(\widetilde Y)^{+r}\cap \widetilde Y'$ contains a convex subcomplex of the form $E'=\prod E_j'$ where each $E_j'$ is either a quasiline or compact, and $\stabilizer_{\hat{A}'}(E')$ acts cocompactly on $E'$ and contains $\hat{A} \cap \hat{A}'$.

Any quasiline in $E$ provides a bi-infinite sequence of nested hyperplanes.
Every hyperplane in this sequence intersects $\widetilde Y'$. Indeed, if some hyperplane in the sequence intersects $(\widetilde Y')^{+r}$ but does not intersect $\widetilde Y'$,
then one side of the sequence would yield hyperplanes arbitrarily far from $\widetilde Y'$,
 and this contradicts that $(\widetilde Y')^{+r}$ lies within a uniform distance of $\widetilde Y'$.
 We deduce that this quasiline corresponds to an entire quasiline of $\widetilde Y'$
 and thus a quasiline of the subproduct $E'$.

  Let $E''$ denote
the subcomplex of $(\widetilde Y)^{+r} \cap (\widetilde Y')^{+r}$ obtained by intersecting it
 with all halfspaces that contain $E\cup E'$.
We now show that  $E'' \subset \neb_s(E)$ and  $E''\subset \neb_s(E')$ for some $s>0$.
Indeed, suppose  $E''\not\subset \neb_s(E)$ for each $s\geq 0$.
Then for each $s$, there is a length~$s$ geodesic $\gamma_s$ in $E''$ that starts at a $0$-cube of $E$, and such that no hyperplane of $E$ intersects $\gamma_s$. Let $\{H_{si}\}_{i=1}^s$ denote the sequence of
hyperplanes dual to $\gamma_s$ and let $\overrightarrow H_{si}$ denote the halfspaces containing $E$.
By definition of $E''$, each  $H_{si}$ either intersects $E'$ or separates $E,E'$.
Note that the number of hyperplanes separating $E, E'$ equals $\dist(E,E')$.
Thus for each $s$, all but $\dist(E,E')$ of the hyperplanes in $\{H_{si}\}_{i=1}^s$ intersect $E'$.
By finite dimensionality there is an upper bound on the number of pairwise crossing hyperplanes,
and so by Ramsey's theorem, for each $t$ there exists $S(t)$,
such that $\gamma_s$ is crossed by $t$ pairwise disjoint hyperplanes whenever $s\geq S(t)$.
We thus obtain arbitrarily long subsequence of hyperplanes that all intersect one of the finitely many factors of $E'=\prod D'_i$.
Since the factors of $E'$ are either finite or quasilines, we see that such a subsequence belongs to a quasiline of $E'$. Thus it belonged to a quasiline of $E$, as explained earlier.
But all hyperplanes of a quasiline of $E$ must cross $E$, which contradicts that no $H_{si}$ crosses $E$.

We now show that  $B=\stabilizer_{\hat{A}}( E)$ and $B'=\stabilizer_{\hat{A}'}( E')$
 are commensurable within $G$. We have already shown that $ E'$ and $ E$ are coarsely equal,
 since each is coarsely equal to $E''$. Let $\Upsilon$ denote the Cayley graph of $G$ with respect
  to a finite generating set. A $G$-equivariant map $\Upsilon\rightarrow \widetilde X$ shows that  $B,B'$ lie within finite neighborhoods of each other within $\Upsilon$.
The right action of $B$ thus stabilizes a finite collection of right cosets of $B'$, and so $B,B'$ are commensurable.

Let $H = B \cap B'$ which is a finite index subgroup of both $B$ and $B'$.
As $\hat{A}\cap \hat{A}'\leqslant B\leqslant \hat{A}$ and $\hat{A}\cap \hat{A}'\leqslant B' \leqslant \hat{A}'$,
we have $H = \hat{A} \cap \hat{A}'$.
Thus $\hat{A}\cap \hat{A}'$ is a finite index subgroup of $B$, hence acts cocompactly on $E$.
The claim then follows from the fact that $B$ is commensurable with a subgroup generated by $S_o$, and that $A \cap A'$ is commensurable to $\hat{A} \cap \hat{A}'$.
\end{proof}

A $\integers^p$ subgroup with a chosen product structure has $p \choose q$ distinct commensurability
classes of $\integers^q$ factor subgroups. We thus have the following corollary to Theorem~\ref{thm:highestIntersectionSubgroups}:
\begin{cor}\label{cor:too many intersection directions}
Suppose $G$ contains a highest free-abelian subgroup $A\cong \integers^p$.
Suppose there are ${p\choose k} +1$ other highest free-abelian subgroups $A_1,\ldots, A_{{p\choose k} +1}$ such that
the subgroups $A\cap A_i$ are pairwise non-commensurable and isomorphic to $\integers^k$.
Then $G$ cannot  act properly and cocompactly on a CAT(0) cube complex.
\end{cor}

We now illustrate Corollary~\ref{cor:too many intersection directions} in a few situations.

\begin{exmp}\label{exmp:4 Z2 subgroups}
We describe an easy example of a group that acts properly on a finite dimensional CAT(0) cube complex
but does not have a finite index subgroup that acts properly and cocompactly on a CAT(0) cube complex.
Consider the group $G$ presented as follows:
$$ G = \langle a,b,r,s,t \mid [a,b], [a,r], [b,s], [ab,t]\rangle$$
Regard $G$ as a multiple HNN extension of $\langle a,b\rangle$ with stable letters $r,s,t$,
we see that $G$ is a ``tubular group'', and deduce that $G$ acts properly on a finite dimensional CAT(0) cube complex by utilizing the \emph{equitable set} $\{ a, b \}$  (see \cite{WiseGerstenRevisited} and \cite{WoodhouseTubularAction}).
However, $G$ does not have a finite index subgroup $G'$ that acts properly and cocompactly on a CAT(0) cube complex.
Indeed, consider the following highest free-abelian subgroups: $A=\langle a,b\rangle$, $R=\langle a,r\rangle$, $S=\langle b,s\rangle$ and $T=\langle ab, t\rangle$.
The intersections $R\cap A$, $S\cap A$, and $T\cap A$ are three pairwise non-commensurable cyclic subgroups of $A$, contradicting Corollary~\ref{cor:too many intersection directions}.

Note that $G$ is a central HNN extension of the 2-dimensional right-angled Artin group
$\langle a,b,r,s \mid [a,b], [a,r], [b,s]\rangle$, and so  the virtually compact
version of Theorem~\ref{thm:Central HNN virtually special} fails without the assumption that $H$ is highest.
\end{exmp}

\begin{exmp}\label{exmp:generic}
 Let $\big\{\langle b_1\rangle,\ldots,\langle b_r\rangle,\langle c_1\rangle,\ldots,\langle c_r\rangle \big\}$
 be a collection of  pairwise incommensurable infinite cyclic subgroups of $\integers^p$,
 and suppose that $r>\frac{p}{2}$.
Let $G$ be the following multiple HNN extension of $\integers^p =\langle a_1,\ldots, a_r\rangle$:
$$G = \langle a_1,\ldots, a_p, t_1,\ldots, t_r \ \mid \  [a_i,a_j]=1, b_k^{t_k}=c_k : 1\leq k \leq r\rangle$$
Then $G$ does not contain a finite index subgroup that acts properly and cocompactly on a CAT(0) cube complex.
Indeed, the subgroups $(\integers^p)^{t_i^{\pm1}}$ intersect $\integers^p$ in the various subgroups
 $\{\langle a_i\rangle, \langle b_i\rangle\}$ and so Corollary~\ref{cor:too many intersection directions} applies.
\end{exmp}

\section{Central HNN extensions of maximal free-abelian subgroups are special} \label{CentralizingExtensions}

This section presumes familiarity with the notions of specialness and canonical completion and retraction. We refer to \cite{HaglundWiseSpecial}.

\begin{lem}\label{lemma:CentralizingCompact Y}
Let $X$ be a virtually special cube complex.
Let $f:Y\rightarrow X$ be a local isometry where $Y$ is a compact nonpositively curved cube complex.
Let $(P, p)$ be a based graph.
Let $Z=\big( X \sqcup  (Y\times P) \big) / \big\{ (y,p)\sim f(y) : \forall y\in Y\big\}$.
Then $Z$ is virtually special.
Moreover, there is a finite special cover $\widehat Z\rightarrow Z$ such that the preimage of $X$ is connected.
\end{lem}

\begin{proof}
 Let $X' \rightarrow X$ be a finite degree special cover of $X$.
Let $Y_i'\rightarrow X'$ be the finitely many elevations of $Y\rightarrow X$.
For each $i$, let $\canon{Y_i'}{X'}$ be the canonical completion of $Y_i' \rightarrow X'$
and identify $Y_i'$ with its image in $\canon{Y_i'}{X'}$.
The canonical retraction $\canon{Y_i'}{X'}\rightarrow Y_i'$ ensures that the maps $Y_i'\rightarrow \canon{Y_i'}{X'}$ are \emph{tidy}
in the sense that they are injective and that no hyperplane $U$ in $\canon{Y_i'}{X'}$ \emph{interosculates} with  $Y_i'$
in the sense that $U$ is dual to an edge in $Y_i'$ and is also dual to an edge that is not in $Y_i'$  but  has an endpoint in $Y_i'$.
Suppose that a hyperplane $U$ intersecting $Y_i'$ were dual to an edge $e$ not in $Y_i'$, but adjacent to a vertex $v \in Y_i'$.
Then $v$ must be adjacent to another edge $e'$ in $Y_i'$ dual to $U$ since the retraction sends hyperplanes to hyperplanes.
This implies a contradiction since the retraction must preserve the orientations of the dual edges, but $U$ cannot self-osculate.

Let $\widehat X$ be a finite degree regular cover of $X$ that factors through each $\canon{Y_i'}{X'}$.
Observe that now all elevations of $Y$ to $\widehat X$ are tidy, since tidiness is stable under covers.
Finally, for each elevation $\widehat Y_j \hookrightarrow \widehat X$ of $Y\rightarrow X$, we adjoin a copy of $\widehat Y_j\times P$.
We thus obtain a cover $\widehat Z \rightarrow Z$.
The specialness of $\widehat Z$ holds due to the tidy embeddings and a case-by-case analysis of its hyperplanes: each hyperplane $W \subset \widehat X$ has $(\widehat Y_j \cap W) \times P$ attached for each $\widehat Y_j$.
Therefore no self-crossings, 1-sided hyperplanes, and no self-osculations are introduced.
The tidiness of each $\canon{\widehat Y_j}{\widehat X}$ guarantees that the new hyperplanes dual to the $P$ factors cannot interosculate with any hyperplane in $\widehat X$, as each elevation factors through some $Y_i'$.
\end{proof}

\begin{rem} Lemma~\ref{lemma:CentralizingCompact Y} can be generalised from the case where $P$ is a graph, to the case where $P$ is a special cube complex.
\end{rem}

We will need the following technical result about right-angled Artin groups.
A subgroup $A \leq G$ is \emph{isolated} if $g^p \in A$ implies that $g \in A$ for some $p \in \mathbb{Z}$.

\begin{lem} \label{lem:isolation_in_raags}
 Let $M$ be an abelian subgroup of a right-angled Artin group $R$.
  Suppose that $M$ is not properly contained in another abelian subgroup, then $M$ is isolated.
\end{lem}

\begin{proof}
 Right-angled Artin groups are biorderable \cite{DuchampThibon92}, therefore if $[g^p, h] = 1$ then $[g,h] =1$.
Indeed, if $ghg^{-1}>h$ then $(ghg^{-1})^n > h^n$ for all $n$, and likewise for $ghg^{-1}<h$.
We conclude that by maximality of $M$, if $g^p \in M$, then $g \in M$.
\end{proof}

\begin{cor} \label{cor:maxRankImpliesHighest}
 If $M$ is a maximal rank abelian subgroup of a right angled Artin group $R$, then it is a highest subgroup of $R$.
\end{cor}

\begin{proof}
If $M$ is virtually contained in a higher rank subgroup $M'$ of $R$, then there exists $g \in M - M'$ with $g^p \in M'$.
 This contradicts the isolation of $M$, by Lemma~\ref{lem:isolation_in_raags}.
\end{proof}

\begin{figure}
  \begin{overpic}[width=.7\textwidth,tics=10]{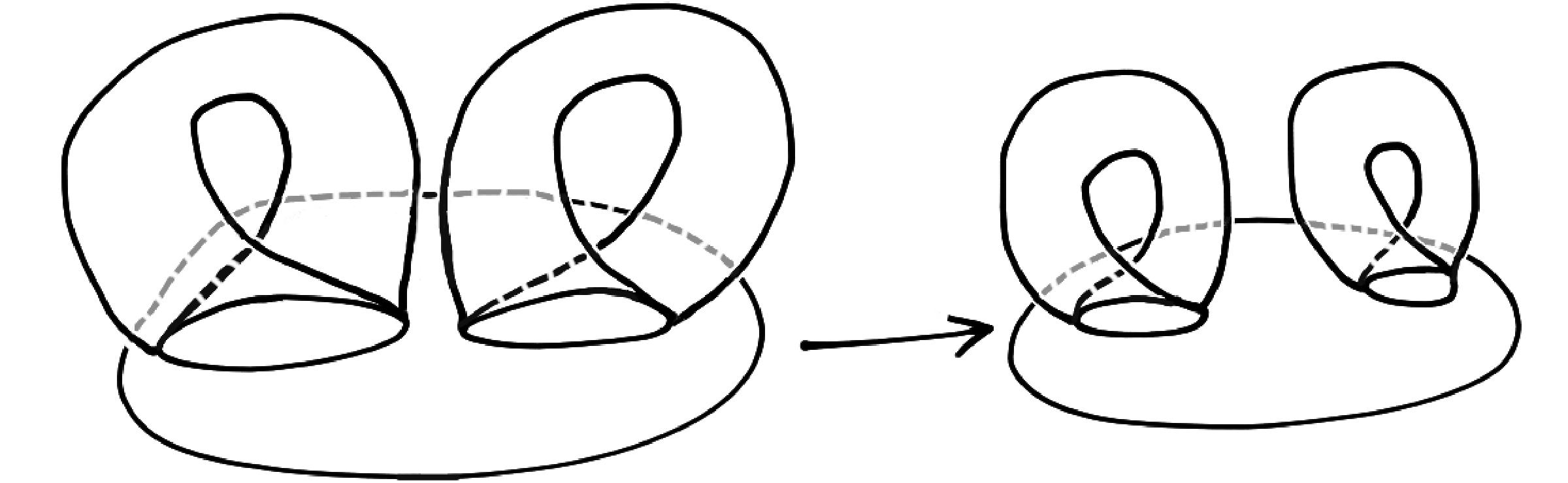}
   \put(27,5){$X'$}
   \put(79,9){$R$}
   \put(-2,29){$Y_1' \times P$}
   \put(24,31){$Y_2' \times P$}
   \put(64,29){$F_1' \times P$}
   \put(82,30){$F_2' \times P$}
  \end{overpic}
  \caption{\label{diag:Another_Diagram}The local isometry $Z \rightarrow S$.}
\end{figure}
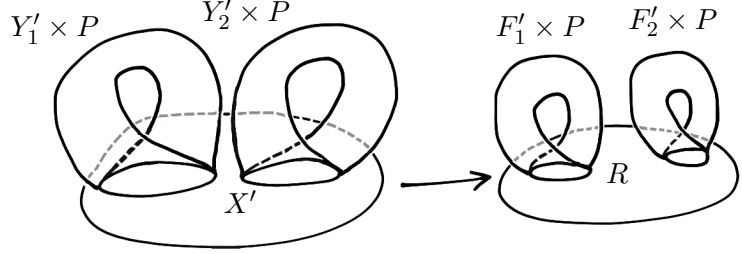

\begin{thm}\label{thm:Central HNN virtually special}
Let $H$ be a finitely generated virtually $[$compact$]$ special group.
Let $A\subset H$ be a highest abelian subgroup.
Let $G=H*_{A^t=A}$ be the HNN extension, where $t$ is the stable letter commuting with $A$,
then $G$ is virtually $[$compact$]$ special.
\end{thm}
\begin{proof}
Let $X'$ be a [compact] nonpositively curved special cube complex such that $\pi_1X'$ is isomorphic to a finite index subgroup of $H$.
We may assume that $X'$ has finitely many hyperplanes since $H$ is finitely generated.
 Consider the local isometry to the associated Salvetti complex $X' \looparrowright R$,
 and note that $R=R(X')$ is compact since $X'$ has finitely many hyperplanes.

 Let $\{g_i \}$ be a finite set of representatives of the double cosets $\{Ag\pi_1X'\}$.
 Let $\{A_i\}$ be the finitely many distinct intersections $\pi_1X'\cap g_i^{-1}Ag_i$.
 Each $A_i$ is highest in $H'$, since $A$ is highest in $H$.
 The subgroup $A_i \hookrightarrow \pi_1R$ is contained in a maximal free-abelian group $\dot B_i \leqslant \pi_1R$, which is highest in $\pi_1R$ by Corollary~\ref{cor:maxRankImpliesHighest}.
As $A_i$ is highest in $H'$ we have $[H'\cap \dot{B}_i:A_i]<\infty$.
The quotient $p_i: \dot B_i / A_i \cong T\oplus \integers^m$ where $|T|<\infty$.
The finite index subgroup $B_i=p_i^{-1}(\integers^m)$ of $\dot B_i$ is still highest in $\pi_1R$ and has the additional property that $H' \cap B_i = A_i$.

By Theorem~\ref{thm:cocompact cubical flat}, for each $i$ there exists a local isometry $F_i\rightarrow R$ with $F_i$ a compact nonpositively curved cube complex,
such that $\pi_1F_i$ maps to $B_i$.
For each $i$, let $Y_i'\rightarrow R$ be the fiber-product of $X'\rightarrow R$ and $F_i\rightarrow R$.
 Note that by possibly replacing $F_i$ with a sufficient convex finite thickening as provided by Lemma~\ref{lem:r thickening}, we can assume that $Y_i'$ is nonempty, so that $\pi_1Y_i'=A_i$.
Let $Z = X'\cup \bigcup(Y_i'\times P) \slash \sim$. Note that $\pi_1Z$ is isomorphic to a finite index subgroup of $G$ since the graph of groups for $\pi_1 Z$ covers the graph of groups of $G$.
Let $S= R\cup \bigcup(F_i\times P) \slash \sim$ be the space obtained from $R$ by attaching the various $F_i\times P$ along $F_i\times \{a\}$ using the map $F_i\rightarrow R$.
See Figure~\ref{diag:Another_Diagram}.

A multiple use of Lemma~\ref{lemma:CentralizingCompact Y} shows that $S$ is virtually special.
There is a local isometry $Z\rightarrow S$ given by the local isometry of $X'$ into $R$ extended along the local isometry $Y_i' \times P \rightarrow F_i \times P$, and hence $Z$ is virtually special.
\end{proof}

\bibliographystyle{alpha}
\bibliography{TPReferences}

\end{document}